\documentclass[reqno,a4paper]{amsart}
\usepackage{amsmath,amstext,amssymb,amsfonts,amscd,amsthm}
\usepackage{geometry}
\usepackage{esint}
\usepackage{hyperref}
\hypersetup{pdfpagemode=FullScreen,colorlinks=true,linkcolor=blue,citecolor=blue,urlcolor=orange}
\usepackage{mathdots,mathrsfs,enumerate}
\usepackage{subfigure}
\usepackage{extarrows}
\usepackage{dsfont}
\usepackage{graphicx,color}
\usepackage{tikz}
\usetikzlibrary{3d,calc,patterns,graphs,arrows}
\usepackage{tikz-cd,array,diagbox}
\numberwithin{equation}{section}

\newenvironment{proof*}{\noindent{\heiti Proof}}{\hfill\qed}
\newtheorem{definition}{Definition}[section]

\newtheorem{lemma}{Lemma}[section]
\newtheorem{theorem}{Theorem}[section]

\newtheorem{proposition}{Proposition}[section]

\allowdisplaybreaks[3]
\usepackage{cite}

\begin{document}
	
	\title{Exterior Dirichlet Problems for Hessian Quotient Equations of Mixed Type}
	
	\author{Yu Lei}
	\address{School of Mathematics and Center for Nonlinear Studies, Northwest University, Xi'an, 710127, PR China}
	\email{leiyu@stumail.nwu.edu.cn}
	
	\author{Zhisu Li}
	\address{School of Mathematics and Center for Nonlinear Studies, Northwest University, Xi'an, 710127, PR China}
	\email{lizhisu@nwu.edu.cn}
	
	\maketitle
	
	\begin{abstract}
		We study the exterior Dirichlet problem for the mixed Hessian quotient equation
		\[
		\frac{\sigma_k(\eta(D^2 u))}{\sigma_l(\eta(D^2 u))} = 1, 
		\]
		where $\eta(M) = (\operatorname{tr} M)I - M$. 
		We establish existence and uniqueness of smooth admissible solutions with prescribed quadratic asymptotics at infinity, 
		and obtain full derivative decay of the remainder. 
		The proof relies on a three-stage subsolution construction.
	\end{abstract}
	
	
	
	\section{Introduction}
	In this paper, we study the exterior Dirichlet problem for the mixed Hessian quotient equation
	\begin{equation}
		\begin{cases}
			\dfrac{\sigma_k(\eta(D^2 u))}{\sigma_l(\eta(D^2 u))} = 1, & \text{in } \mathbb{R}^n \setminus \overline{D}, \\[6pt]
			u = \varphi, & \text{on } \partial D,
		\end{cases}
		\label{eq:main}
	\end{equation}
	where $D \subset \mathbb{R}^n$ is a bounded domain and $\varphi \in C^\infty(\partial D)$.
	
	Let $\lambda(D^2 u) = (\lambda_1, \dots, \lambda_n)$ denote the eigenvalues of the Hessian matrix $D^2 u$ of a smooth function $u$ on $\mathbb{R}^n$. 
	For $1 \le k \le n$, the $k$-th elementary symmetric function is defined by
	\[
	\sigma_k(\lambda) = \sum_{1 \le i_1 < \dots < i_k \le n} \lambda_{i_1} \cdots \lambda_{i_k}.
	\]
	For a symmetric matrix $M$, we write $\sigma_k(M) = \sigma_k(\lambda(M))$.
	Define the linear operator $\eta$ on symmetric matrices by
	\[
	\eta(M) = (\operatorname{tr} M) I - M.
	\]
	If $\lambda(D^2 u) = (\lambda_1, \dots, \lambda_n)$, then $\eta(D^2 u)$ has eigenvalues
	\[
	\eta_i = \sum_{j \neq i} \lambda_j, \qquad i = 1, \dots, n.
	\]
	
	The operator $\sigma_k(\eta(D^2 u))$ arises naturally in various geometric contexts. When $l=0$, it appears as a ``form-type'' Calabi--Yau equation proposed by Fu--Wang--Wu~\cite{FW10} and studied by Tosatti--Weinkove~\cite{TW17} on K\"ahler manifolds. In prescribed curvature problems, the same mixed Hessian structure appears through
	\[
	\eta_{ij} = H g_{ij} - h_{ij},
	\]
	where $h_{ij}$ and $H$ denote the second fundamental form and the mean curvature of a hypersurface; see Chu--Jiao~\cite{CJ21}.
	
	The mixed quotient equation $\sigma_k(\eta(D^2 u)) / \sigma_l(\eta(D^2 u)) = 1$ is a natural generalization of the Hessian quotient equation
	\[
	\frac{\sigma_k(\lambda(D^2 u))}{\sigma_l(\lambda(D^2 u))} = 1,
	\]
	which has been extensively studied in the literature; when $k=n$ and $l=0$, it further reduces to the Monge--Amp\`ere equation $\det(D^2 u) = 1$. 
	
	The mixed quotient equation $$\dfrac{\sigma_k(\eta(D^2 u))}{\sigma_l(\eta(D^2 u))} = f$$ was studied by Chen--Dong--Han~\cite{CDH22}, who established interior Hessian estimates and Pogorelov type estimates, which in particular imply a Liouville theorem. The Dirichlet problem on bounded domains for the equation with a gradient-dependent right-hand side was later considered by Chen--Tu--Xiang~\cite{CTX23} on Riemannian manifolds.  However, is restricted to the range $l+2 \le k \le n$, a condition imposed by the algebraic structure of the quotient operator in their boundary estimates. In this paper, we overcome this limitation and establish existence and uniqueness for the exterior problem with no restriction on the gap $k-l$. The critical case $l = n-1$ is treated separately using a refined expansion of $\sigma_{n-1}$ for block matrices, which may be of independent interest.
	
	The exterior Dirichlet problem for fully nonlinear elliptic equations has a rich history. For the Monge--Amp\`ere equation $\det(D^2u)=1$, Caffarelli and Li~\cite{CL03} proved that any convex viscosity solution outside a bounded domain must be asymptotic to a quadratic polynomial at infinity, and established existence and uniqueness for the exterior Dirichlet problem. This result was later extended by Bao--Li--Li~\cite{BLL14} to the $k$-Hessian equation $\sigma_k(\lambda(D^2u))=1$, and by Li and Li~\cite{LL18} to the Hessian quotient equation $\sigma_k(\lambda(D^2u))/\sigma_l(\lambda(D^2u))=1$. However, all these works require the domain $D$ to be strictly convex, and the solutions obtained are viscosity solutions.
	Recently, Li and Xiao~\cite{LYX26} proved the existence of smooth solutions for the exterior Dirichlet problem of the $k$-Hessian equation on non-convex domains, extending the previous result of Bao--Li--Li~\cite{BLL14} for strictly convex domains.The special Lagrangian equation
	\[
	\sum_{i=1}^n \arctan \lambda_i(D^2u) = \Theta,
	\]
	it can be written as the Hessian quotient equation when $l=1$ and $k=n=3$:
	\[
	\det(D^2u)= \triangle u.
	\]
	The exterior Dirichlet problem for the special Lagrangian equation was studied by  Li~\cite{Li19} and Li--Li--Yuan~\cite{LLY20}, who established quadratic asymptotics for solutions of the special Lagrangian equation with supercritical phases $|\Theta|>(n-2)\pi/2$ in exterior domains.
	
	For Hessian quotient equations and special Lagrangian equations in exterior domains, the subsolution construction typically yields a decay rate $|x|^{m-2}$ with $m\in(2,n]$ depending on the algebraic structure of the equation, and requires additional technical restrictions on the asymptotic matrix $A$ (see Li and Li \cite{LL18} for Hessian quotient equations and Li \cite{Li19} for special Lagrangian equations). In contrast, by applying the exterior Liouville theorem of Li--Li--Yuan \cite{LLY20}, we establish the optimal decay $|x|^{n-2}$ for the mixed Hessian quotient equation \eqref{eq:main}, with full decay estimates for all derivatives of the remainder, under the natural condition $A\in\mathcal{A}_{k,l}$ without any further assumptions.
	
	\begin{theorem}
		Let $D \subset \mathbb{R}^n$be a bounded $C^\infty$ domain with boundary data $\varphi \in C^\infty(\partial D)$,$n\geq3$. Then for any given $A \in \mathcal{A}_{k,l}$ with  $0\leq l<k<n$ and any $b\in \mathbb{R}^n$, where
		\[
		\mathcal{A}_{k,l} = \{\, A \mid A \text{ is a real } n\times n \text{ symmetric positive definite matrix with } \sigma_k(\eta(A)) = \sigma_l(\eta(A)) \,\},
		\]
		there exists $c_* = c_*(n,k,l,b,A,\partial D,\|\varphi\|_{C^2(\partial D)})$ such that for every $c > c_*$, the exterior Dirichlet problem \eqref{eq:main}	admits a unique strictly $\widetilde{\Gamma}_k$-admissible solution
		\[
		u \in C^\infty(\mathbb{R}^n \setminus \overline{D}) \cap C^0(\mathbb{R}^n \setminus D),
		\]
		satisfying
		\[
		\lim_{|x| \to \infty} |x|^{n-2} \left| u(x) - \left( \frac12 x^T A x + b \cdot x + c \right) \right| < \infty.
		\]
		
	\end{theorem}
	For every affine function $\ell(x)=b\cdot x+d$, one has $\eta(D^2(u-\ell))=\eta(D^2 u)$. Hence, after replacing $u$ by $u-b\cdot x$ and $\varphi$ by $\varphi-b\cdot x$, one may reduce to the case $b=0$. All constructions  are carried out under this normalization.
	
	\begin{theorem}
		Let $D \subset \mathbb{R}^n$ be a bounded strictly
		$\eta$-$1$-convex $C^\infty$ domain with boundary data $\varphi \in C^\infty(\partial D)$. Then for any given $A \in \mathcal{A}_{n,l}$ with  $0\leq l<k=n$ and any $b\in \mathbb{R}^n$, there exists $c_* = c_*(n,l,b,A,\partial D,\|\varphi\|_{C^2(\partial D)})$ such that for every $c > c_*$, the exterior Dirichlet problem  \eqref{eq:main} admits a unique strictly $\widetilde{\Gamma}_n$-admissible solution
		\[
		u \in C^\infty(\mathbb{R}^n \setminus \overline{D}) \cap C^0(\mathbb{R}^n \setminus D),
		\]
		satisfying
		\[
		\lim_{|x| \to \infty} |x|^{n-2} \left| u(x) - \left( \frac12 x^T A x + b \cdot x + c \right) \right| < \infty.
		\]
	\end{theorem}

	The paper is organized as follows. In section 2 we recall necessary preliminaries. In  section 3, we construct a global subsolution for \eqref{eq:main} through a three-stage argument. The overall strategy is inspired by Li and Xiao \cite{LYX26}. We first construct a near-boundary subsolution in \(E_1 \setminus \overline{D}\), then a far-field subsolution in \(\mathbb{R}^n \setminus E_\tau\), and finally an intermediate layer connecting them. The gluing of the three pieces is achieved via the regularized maximum technique of Caffarelli and Li \cite{CL03}.  In Section 4 we solve approximating Dirichlet problems and derive uniform a priori estimates. Section 5 contains the proof of the main theorem, including the asymptotic expansion and uniqueness.
	
	\section{Preliminaries}
	In this preparatory section, we recall standard notation, definitions, and preliminary results.
	
	\subsection{Notation and definitions}
	The following notation and definitions will be used throughout this paper.
	Let $$s(x) := \frac12 x^T A x,$$and for $l > 0$, define $$E_l := \{ s(x) < l \}.$$ 
	
	Choose $\tau_1 > 1$ such that $\overline{D} \subset E_{\tau_1}$, and set $E_1 := E_{\tau_1}.$
	
	Let $$Q_c(x) := s(x) + c.$$
	
	Let $d(x)$ be the signed distance to $\partial D$, positive outside $D$, and let $\pi(x)$ be the nearest point projection onto $\partial D$, defined in a tubular neighborhood $U_{\delta_*} = \{ 0 < d(x) < \delta_* \}$.
	
	For any $\lambda = (\lambda_1, \dots, \lambda_n) \in \mathbb{R}^n$, the $k$-th elementary symmetric function is defined by
	\[
	\sigma_k(\lambda) = \sum_{1 \le i_1 < \dots < i_k \le n} \lambda_{i_1} \cdots \lambda_{i_k}, \qquad 1 \le k \le n.
	\]
	We set $\sigma_0(\lambda) = 1$ and $\sigma_k(\lambda) = 0$ for $k > n$. 
	For convenience, we denote by $\sigma_k(\lambda \mid i)$ the symmetric function with $\lambda_i = 0$, and by $\sigma_k(\lambda \mid ij)$ the symmetric function with $\lambda_i = \lambda_j = 0$.

	\begin{definition}
		For $1 \le k \le n$, the G{\aa}rding cone $\Gamma_k$ is defined by
		\begin{equation}
			\Gamma_k = \{ \lambda \in \mathbb{R}^n : \sigma_j(\lambda) > 0 \text{ for all } 1 \le j \le k \}.
		\end{equation}
		It is well known that $\Gamma_k$ is an open convex cone with vertex at the origin, and
		\[
		\Gamma_1 \supset \Gamma_2 \supset \cdots \supset \Gamma_n = \Gamma^+,
		\]
		where $\Gamma^+ = \{ \lambda \in \mathbb{R}^n : \lambda_i > 0 \text{ for all } i \}$.
		The admissible cone for the mixed Hessian operator is defined by
		\begin{equation}
			\widetilde{\Gamma}_k = \{ \lambda \in \mathbb{R}^n : \sigma_i(\eta(\lambda)) > 0 \text{ for all } 1 \le i \le k \}.
		\end{equation}
	\end{definition}

	\begin{definition}
		Let $U \subset \mathbb{R}^n$ be open and $2 \leq k \leq n$. A function $u \in C^2(U)$ is called strictly $\widetilde{\Gamma}_k$-admissible if
		\[
		\lambda(D^2 u(x)) \in \widetilde{\Gamma}_k \quad \text{for every } x \in U.
		\]
		It is called $\widetilde{\Gamma}_k$-admissible if
		\[
		\lambda(D^2 u(x)) \in \overline{\widetilde{\Gamma}}_k \quad \text{for every } x \in U.
		\]
		Equivalently,
		\[
		\sigma_i(\eta(D^2 u)(x)) > 0 \quad (\text{respectively } \geq 0) \quad \text{for all } 1 \leq i \leq k \text{ and } x \in U.
		\]
	\end{definition}

	\begin{definition}
		Let $\kappa = (\kappa_1, \ldots, \kappa_{n-1})$ be the principal curvature vector of $\partial D$ with respect to the outward unit normal, and let
		\[
		H := \sigma_1(\kappa) = \sum_{\alpha=1}^{n-1} \kappa_\alpha, \qquad
		\eta(\kappa) := \left( \sum_{\beta \neq 1} \kappa_\beta, \ldots, \sum_{\beta \neq n-1} \kappa_\beta \right) = (H - \kappa_1, \ldots, H - \kappa_{n-1}).
		\]
		For the boundary curvature vector $\kappa \in \mathbb{R}^{n-1}$, we use the same notation $\Gamma_m$ and $\widetilde{\Gamma}_m$ for the corresponding cones defined in dimension $n-1$.
		For $1 \leq m \leq n-1$, we say that $\partial D$ is strictly $\eta$-$m$-convex if
		\[
		\kappa \in \widetilde{\Gamma}_m \quad \text{or} \quad \eta(\kappa) \in \Gamma_m \quad \text{on } \partial D.
		\]
	\end{definition}
	
	we show some properties about the elementary symmetric function, 
	
	\begin{proposition}\label{prop}
		The following properties hold:
		\begin{enumerate}
			\item $\widetilde{\Gamma}_k$ are convex cones and
			\[
			\Gamma_1 = \widetilde{\Gamma}_1 \supset \widetilde{\Gamma}_2 \supset \cdots \supset \widetilde{\Gamma}_n \supset \Gamma_2.
			\]
			
			\item If $\lambda = (\lambda_1, \ldots, \lambda_n) \in \widetilde{\Gamma}_k$, then
			\[
			\frac{\partial \left[ \frac{\sigma_k(\eta)}{\sigma_l(\eta)} \right]}{\partial \lambda_i}
			= \sum_{p \neq i} \frac{\sigma_{k-1}(\eta|p)\sigma_l(\eta) - \sigma_k(\eta)\sigma_{l-1}(\eta|p)}{\sigma_l(\eta)^2},
			\]
			and
			\[
			\frac{\partial \left[ \frac{\sigma_k(\eta)}{\sigma_l(\eta)} \right]}{\partial \lambda_i}
			\geq \frac{n(k-l)}{k(n-l)} \sum_{p \neq i} \frac{\sigma_{k-1}(\eta|p)\sigma_l(\eta|p)}{\sigma_l(\eta)^2},
			\]
			for any $i = 1, 2, \ldots, n$, where $0 \leq l < k \leq n$.
			
			\item If $\lambda = (\lambda_1, \ldots, \lambda_n) \in \widetilde{\Gamma}_k$, then $\left[ \frac{\sigma_k(\eta)}{\sigma_l(\eta)} \right]^{\frac{1}{k-l}}$ $(0 \leq l < k \leq n)$ is concave with respect to $\lambda$. Hence, for any $(\xi_1, \ldots, \xi_n)$, we have
			\[
			\sum_{i,j} \frac{\partial^2 \left[ \frac{\sigma_k(\eta)}{\sigma_l(\eta)} \right]}{\partial \lambda_i \partial \lambda_j} \xi_i \xi_j
			\leq
			\left(1 - \frac{1}{k-l}\right)
			\frac{
				\left(
				\sum_i \frac{\partial \left[ \frac{\sigma_k(\eta)}{\sigma_l(\eta)} \right]}{\partial \lambda_i} \xi_i
				\right)^2
			}{
				\frac{\sigma_k(\eta)}{\sigma_l(\eta)}
			}.
			\]
		\end{enumerate}
	\end{proposition}
	\begin{proof}
		All the properties are well known. For example, see Proposition 2.6 in \cite{CDH22}.
	\end{proof}
	
	We write
	
	\[
	F(M) := 
	\begin{cases} 
		\left[ \dfrac{\sigma_k(\eta(M))}{\sigma_l(\eta(M))} \right]^{1/(k-l)}, & 2 \le k \le n-1, \\[1.2ex]
		\left[ \dfrac{\det(\eta(M))}{\sigma_l(\eta(M))} \right]^{1/(n-l)}, & k = n.
	\end{cases}
	\]
	With this notation, the equation can be written uniformly as
	\begin{equation}
		F(D^2 u) = 1.
	\end{equation}
	
	\begin{lemma}\label{lem:eta_composition}
		Let $U \subset \mathbb{R}^n$ be open, let $w \in C^2(U)$, and let $u = \Phi(w)$ with $\Phi \in C^2(\mathbb{R})$. Then
		\[
		\eta(D^2u) = \Phi'(w)\eta(D^2w) + \Phi''(w)\bigl(|Dw|^2 I - Dw \otimes Dw\bigr).
		\]
		If $|Dw| = 1$ in $U$, then
		\[
		\bigl(|Dw|^2 I - Dw \otimes Dw\bigr)Dw = 0, \qquad
		\bigl(|Dw|^2 I - Dw \otimes Dw\bigr)\xi = \xi \quad \text{for every } \xi \perp Dw.
		\]
		In particular, the matrix $|Dw|^2 I - Dw \otimes Dw$ has eigenvalues $1, \ldots, 1, 0$.
	\end{lemma}
	
	\begin{proof}
		For $1 \le i, j \le n$,
		\[
		u_i = \Phi'(w)w_i, \qquad u_{ij} = \Phi''(w)w_iw_j + \Phi'(w)w_{ij}.
		\]
		Summing over $i = j$ gives
		\[
		\Delta u = \Phi''(w)|Dw|^2 + \Phi'(w)\Delta w.
		\]
		Therefore,
		\[
		\begin{aligned}
			(\eta(D^2u))_{ij}
			&= (\Delta u)\delta_{ij} - u_{ij} \\
			&= \bigl(\Phi''(w)|Dw|^2 + \Phi'(w)\Delta w\bigr)\delta_{ij}
			- \Phi''(w)w_iw_j - \Phi'(w)w_{ij} \\
			&= \Phi'(w)\bigl((\Delta w)\delta_{ij} - w_{ij}\bigr)
			+ \Phi''(w)\bigl(|Dw|^2\delta_{ij} - w_iw_j\bigr).
		\end{aligned}
		\]
		This is precisely
		\[
		\eta(D^2u) = \Phi'(w)\eta(D^2w) + \Phi''(w)\bigl(|Dw|^2 I - Dw \otimes Dw\bigr).
		\]
	\end{proof}
	
	\begin{lemma}\label{lem:tubular}
		Assume $\partial D$ is $C^\infty$. Then there exists $\delta_* > 0$ such that
		\[
		U_{\delta_*} := \{x \in \mathbb{R}^n \setminus \overline{D} : 0 < d(x) < \delta_*\}, \qquad d(x) := \operatorname{dist}(x, \partial D),
		\]
		is a tubular neighborhood of $\partial D$, and the nearest point projection
		\[
		\pi : U_{\delta_*} \to \partial D
		\]
		is smooth. Moreover, for every $x \in U_{\delta_*}$ there exists an orthonormal frame $\{e_1, \ldots, e_{n-1}, \nu\}$ with $\nu = Dd(x)$ such that
		\begin{equation}
			D^2d(e_i, e_j) = \frac{\kappa_i(\pi(x))}{1 + \kappa_i(\pi(x))d(x)}\delta_{ij}, \qquad D^2d(\nu, \cdot) = 0.
		\end{equation}
		Here $\kappa_1, \ldots, \kappa_{n-1}$ are the principal curvatures of $\partial D$ with respect to its outward unit normal.
	\end{lemma}
	
	\begin{lemma}\label{lem:regularized_max}
		Let $V \subset \mathbb{R}^n$ be open and let $v_1, \ldots, v_N \in C^2(V)$ be strictly $\widetilde{\Gamma}_k$-admissible subsolutions of
		\[
		F(D^2 v) \ge 1 \quad \text{in } V.
		\]
		Then any finite iterated regularized maximum of $v_1, \ldots, v_N$, with sufficiently small smoothing parameter, is again a smooth strictly $\widetilde{\Gamma}_k$-admissible subsolution in $V$. Moreover, if on some set one $v_j$ dominates all the others by a fixed positive gap, then the regularized maximum agrees with $v_j$ there.
	\end{lemma}
	
	\begin{proof}
		This is the standard regularized maximum construction from Caffarelli--Li \cite{CL03}. For two functions, choose $\chi \in C^\infty(\mathbb{R})$ such that
		\[
		\chi(t) = 0 \text{ for } t \le -1, \qquad
		\chi(t) = t \text{ for } t \ge 1, \qquad
		0 \le \chi' \le 1, \qquad
		\chi'' \ge 0,
		\]
		and set
		\[
		M_\delta(a,b) := b + \delta \chi\left(\frac{a-b}{\delta}\right).
		\]
		Then 
		\[
		D^2 M_\delta(a,b)
		= \theta D^2 a + (1-\theta) D^2 b
		+ \delta^{-1} \chi''\left(\frac{a-b}{\delta}\right) (Da - Db) \otimes (Da - Db),
		\]
		where $\theta = \chi'\left(\frac{a-b}{\delta}\right) \in [0,1]$.
		Applying $\eta(M) = (\operatorname{tr}M)I - M$ gives
		\[
		\eta(D^2 M_\delta(a,b))
		= \theta \eta(D^2 a) + (1-\theta) \eta(D^2 b)
		+ \delta^{-1} \chi''\left(\frac{a-b}{\delta}\right)
		\bigl(|Da - Db|^2 I - (Da - Db) \otimes (Da - Db)\bigr).
		\]
		Furthermore, since $F$ is concave and positively homogeneous on $\widetilde{\Gamma}_k$,
		\[
		F(\theta D^2 a + (1-\theta) D^2 b) \ge \theta F(D^2 a) + (1-\theta) F(D^2 b) \ge 1.
		\]
		Since $F$ is also monotone on $\widetilde{\Gamma}_k$ and $\delta^{-1}\chi'' N \in \overline{\Gamma}_n$, we obtain
		\[
		F(D^2 M_\delta(a,b)) = F(\theta D^2 a + (1-\theta) D^2 b + \delta^{-1}\chi'' N) \ge 1.
		\]
		Thus $M_\delta(a,b)$ is a $\widetilde{\Gamma}_k$-admissible subsolution and where 
		$$	N := 
	|Da - Db|^2 I - (Da - Db) \otimes (Da - Db).$$
	\end{proof}
	
	\begin{lemma}[Comparison principle]
		Let $\Omega \subset \mathbb{R}^n$ be a bounded domain. Suppose $u, v \in C^2(\Omega) \cap C^0(\overline{\Omega})$ are strictly $\widetilde{\Gamma}_k$-admissible and satisfy
		\[
		F(D^2 u) \le F(D^2 v) \quad \text{in } \Omega,
		\]
		and
		\[
		u \ge v \quad \text{on } \partial \Omega.
		\]
		Then $u \ge v$ in $\Omega$.
	\end{lemma}
	
	\begin{proof}
		The proof is divided into two steps.
		
		\paragraph{Step 1: Strict inequality.}
		We first prove the lemma under the stronger assumption
		\[
		F(D^2 u) < F(D^2 v) \quad \text{in } \Omega.
		\]
		We argue by contradiction. Define $w = u - v$ and suppose there exists $x_0 \in \Omega$ such that
		\[
		w(x_0) = \min_{\Omega} w < 0.
		\]
		At this point, we have $D^2 w(x_0) \ge 0$, i.e., $D^2 u(x_0) \ge D^2 v(x_0)$.
		Since $\eta$ is linear, applying it to both sides gives
		\[
		\eta(D^2 u(x_0)) \ge \eta(D^2 v(x_0)).
		\]
		By the superadditivity of $F$,
		\[
		F(D^2u(x_0)) = F(D^2v(x_0) + D^2w(x_0)) \ge F(D^2v(x_0)) + F(D^2w(x_0)).
		\]
		This contradicts the strict inequality $F(D^2 u) < F(D^2 v)$ at $x_0$. Therefore, $w \ge 0$ in $\Omega$, i.e., $u \ge v$ in $\Omega$.
		
		\paragraph{Step 2: General case.}
		Now assume only the non-strict inequality
		\[
		F(D^2 u) \le F(D^2 v) \quad \text{in } \Omega.
		\]
		For $\epsilon > 0$, define
		\[
		u_\epsilon(x) = u(x) + \epsilon (C - |x - x_0|^2),
		\]
		where $x_0 \in \Omega$ is fixed and $C = \max_{x \in \partial \Omega} |x - x_0|^2$ so that $u_\epsilon \ge v$ on $\partial \Omega$.
		
		Since $D^2 u_\epsilon = D^2 u - 2\epsilon I$ and $F$ is strictly monotone, we have
		\[
		F(D^2 u_\epsilon) < F(D^2 u) \le F(D^2 v).
		\]
		Thus $F(D^2 u_\epsilon) < F(D^2 v)$ in $\Omega$. Applying Step 1 to $u_\epsilon$ and $v$, we get $u_\epsilon \ge v$ in $\Omega$. Letting $\epsilon \to 0$ yields $u \ge v$ in $\Omega$.
	\end{proof}
	
	\section{Construction of a global subsolution}
	
	In this section, we construct a global subsolution for the quotient equation via a three-stage argument.

	\begin{proposition}[Near boundary barrier for the mixed hessian quotient equation]\label{near}
		There exist $T > 1$, $\gamma \in \mathbb{R}$, and 
		\[
		u_{\mathrm{near}} \in C^\infty(E_1 \setminus \overline{D}) \cap C^0(\overline{E_1} \setminus D)
		\]
		with the following properties.
		\begin{enumerate}
			\item[(i)] If $0 \le l < k \le n-1$, then for some $\alpha > 1$,
			\[
			u_{\mathrm{near}}(x) = \varphi(\pi(x)) + e^{\alpha d(x)} - 1 \quad \text{near } \partial D.
			\]
			\item[(ii)] If $0\leq l<k = n$, then for some $K, B > 0$,
			\[
			u_{\mathrm{near}}(x) = \varphi(\pi(x)) + K d(x) + \frac{B}{2} d(x)^2 \quad \text{near } \partial D.
			\]
		\end{enumerate}
		In both cases,
		\[
		u_{\mathrm{near}} = Ts + \gamma \quad \text{near } \partial E_1, \qquad
		u_{\mathrm{near}} = \varphi \quad \text{on } \partial D,
		\]
		and $u_{\mathrm{near}}$ is a strictly $\widetilde{\Gamma}_k$-admissible subsolution of the quotient equation in $E_1 \setminus \overline{D}$.
	\end{proposition}
	
	\begin{proof}
		We divide the proof into two cases according to whether $k < n$ or $k = n$.
		
		\medskip
		\noindent\textbf{Case 1: $0 \le l < k \le n-1$.} 
		Let
		\[
		w(x) = \varphi(\pi(x)) + e^{\alpha d(x)} - 1, \qquad x \in U_{\delta_*}.
		\]
		Using the linearity of $\eta$ and the chain rule,
		\[
		\eta(D^2 w) = \eta(D^2(\varphi(\pi(x)))) + \alpha e^{\alpha d} \eta(D^2 d) + \alpha^2 e^{\alpha d} (I - Dd \otimes Dd).
		\]
		Dividing both sides by $\alpha^2 e^{\alpha d}$, we obtain
		\[
		\alpha^{-2} e^{-\alpha d} \eta(D^2 w)
		= \alpha^{-2} e^{-\alpha d} \eta(D^2(\varphi(\pi(x))))
		+ \alpha^{-1} \eta(D^2 d)
		+ (I - Dd \otimes Dd).
		\]
		In the principal frame, $I - Dd \otimes Dd = \operatorname{diag}(1,\dots,1,0) =: A_0$.
		
		As $\alpha \to \infty$,
		\[
		\alpha^{-2} e^{-\alpha d} \eta(D^2 w) \to A_0,
		\]
		and $\sigma_k(A_0) = C_{n-1}^{k}> 0$.
		
		Using the homogeneity of $\sigma_k$ and $\sigma_l$,
		\[
		F(D^2 w) = \alpha^2 e^{\alpha d} \cdot \left( \frac{\sigma_k(\alpha^{-2} e^{-\alpha d} \eta(D^2 w))}{\sigma_l(\alpha^{-2} e^{-\alpha d} \eta(D^2 w))} \right)^{1/(k-l)}.
		\]
		As $\alpha \to \infty$, the fraction converges to
		\[
		\left( \frac{\sigma_k(A_0)}{\sigma_l(A_0)} \right)^{1/(k-l)} = 
		\left( \frac{C_{n-1}^{k}}{C_{n-1}^{l}} \right)^{1/(k-l)} > 0.
		\]
		Thus for sufficiently large $\alpha$,
		\[
		F(D^2 w) \ge \frac12 \alpha^2 e^{\alpha d} \cdot \left( \frac{\binom{n-1}{k}}{\binom{n-1}{l}} \right)^{1/(k-l)} > 1.
		\]
		Hence $w$ is a strict subsolution near $\partial D$.
		
		\medskip
		\noindent\textbf{Case 2: $0\leq l+1<k = n$.} 
		Let
		\[
		w(x) = \varphi(\pi(x)) + K d(x) + \frac{B}{2} d(x)^2, \qquad x \in U_{\delta_*},
		\]
		where $K, B > 0$ are constants to be chosen. Using the linearity of $\eta$ and the chain rule,
		\[
		\eta(D^2 w) = \eta(D^2(\varphi(\pi(x)))) + (K + B d) \eta(D^2 d) + B (I - Dd \otimes Dd). 
		\]
		In the principal frame, write
		\begin{equation}
			D^2 d = \operatorname{diag}(\lambda_1, \dots, \lambda_{n-1}, 0), \qquad
			\lambda_\alpha = \frac{\kappa_\alpha}{1 + \kappa_\alpha d},
		\end{equation}
		where $\kappa_\alpha$ are the principal curvatures of $\partial D$. Let
		\[
		H_d := \sum_{\alpha=1}^{n-1} \lambda_\alpha.
		\]
		Then
		\[
		\eta(D^2 d) = \operatorname{diag}(H_d - \lambda_1, \dots, H_d - \lambda_{n-1}, H_d).
		\]
		
		Since $1\leq l+1<k = n$, by the assumption of strict mean convexity of $\partial D$ in Theorem 1.2, there exists $h_0 > 0$ such that the mean curvature $H := \sum_{\alpha=1}^{n-1} \kappa_\alpha$ satisfies
		\[
		H \ge h_0 \quad \text{on } \partial D.
		\]
		After shrinking $\delta_*$ if necessary, we may assume in $U_{\delta_*}$ that
		\begin{equation}
			H_d \ge \frac{h_0}{2}, \qquad 
			d |H_d - \lambda_\alpha| \le \frac{1}{2} \quad (1 \le \alpha \le n-1). 
		\end{equation}
		Set
		\[
		C_0 := \| \eta(D^2(\varphi(\pi(x)))) \|_{C^0(\overline{U_{\delta_*}})}, \quad
		C_1 := \sup_{\substack{x \in U_{\delta_*} \\ 1 \le \alpha \le n-1}} |H_d(x) - \lambda_\alpha(x)|,\quad 
		C_2 := \sup_{x\in U_{\delta_*}} |H_d(x)|.
		\]
		
		To estimate the eigenvalues of $\eta(D^2 w)$, we divide them into the first $n-1$ terms and the $n$-th term.
		In the principal frame the matrix
		\[
		(K + B d) \eta(D^2 d) + B (I - Dd \otimes Dd)
		\]
		has tangential eigenvalues
		\[
		B + (K + B d)(H_d - \lambda_\alpha)
		= B + K(H_d - \lambda_\alpha) + B d (H_d - \lambda_\alpha).
		\]
		Then,
		\[
		\frac{B}{2} - K C_1
		\le
		B + (K + Bd)(H_d - \lambda_\alpha)
		\le
		\frac{3B}{2} + K C_1,
		\]
		Its normal eigenvalue is
		$$K \dfrac{h_0}{2} \le (K + Bd)H_d \le K C_2 + B\delta_* C_2.$$
		Moreover, 
		Let \(\eta_1,\dots,\eta_n\) be the eigenvalues of \(\eta(D^2w)\). 
		Thus, we have
		\begin{equation}
			\frac{B}{2} - K C_1 - C_0
			\le \eta_\alpha \le
			\frac{3B}{2} + K C_1 + C_0,
			\qquad \alpha = 1,\dots,n-1,
		\end{equation}
		and
		\begin{equation}
			K \frac{h_0}{2} - C_0
			\le \eta_n \le
			K C_2 + B\delta_* C_2 + C_0.
		\end{equation}
		First choose $K_1$ and $B_1$ such that $$K_1 \frac{h_0}{2} - C_0>2,\quad \frac{B_1}{2} - K_1 C_1 - C_0>2.$$
		From the above estimates, since \(l<n-1\), we may choose \(B_2\) sufficiently large so that
		\begin{equation}
			\frac{\sigma_n(\eta(D^2 w))}{\sigma_l(\eta(D^2 w))}
			\geq 
			\frac{
				\left( \frac{B_2}{2} - K_1 C_1 - C_0 \right)^{n-1}
				\left( K_1 \frac{h_0}{2} - C_0 \right)
			}{
				C_{n-1}^l 
				\left( \frac{3B_2}{2} + K_1 C_1 + C_0 \right)^l
				+ O(B^{l-1}_2)
			} > 1.
		\end{equation}
		Finally set $B>B_1$, therefore
		\begin{equation}
			F(D^2 w) = \left( \frac{\sigma_n(\eta(D^2 w))}{\sigma_l(\eta(D^2 w))} \right)^{1/(n-l)} > 1.
		\end{equation}
		when $l=n-1$, $k=n$, taking $B_1=mK_1$ and choose $m > 2C_1 + \frac{2(C_0+2)}{K_1}$  so that
		$$\frac{mK_1}{2} - K_1 C_1 - C_0>2.$$
		can choose  $K>K_1$ sufficiently large  such that
		\begin{equation}
			\frac{\sigma_n(\eta(D^2 w))}{\sigma_{n-1}(\eta(D^2 w))}
			\geq 
			\frac{
				\left( \frac{mK}{2} - K C_1 - C_0 \right)^{n-1}
				\left( K \frac{h_0}{2} - C_0 \right)
			}{
				\left( \frac{3mK}{2} + K C_1 + C_0 \right)^{n-1}
				+ O(K^{n-2})
			} > 1.
		\end{equation}
		Thus $w$ is a strict subsolution near $\partial D$ in the case $0\leq l<k = n$.
		
		We next continue $w$ across the fixed ring between $\partial D$ and $\partial E_1$ by the standard finite covering gluing argument of Lemma \ref{lem:regularized_max}. For each $\xi \in \partial D$, choose $z_\xi = \xi - r\nu(\xi) \in D$ with fixed $r > 0$, where $\nu(\xi)$ is the outward unit normal to $\partial D$ at $\xi$, and consider
		\begin{equation}
			q_\xi(x) := \frac{t}{2} \left( x - z_\xi \right)^T A \left( x - z_\xi \right) + a_\xi, \qquad q_\xi(\xi) = w(\xi) - 1.
		\end{equation}
		Because $A$ is positive definite, 
		\[
		\frac{d^2}{ds^2} q_\xi(\xi+s\nu(\xi))
		=
		t\,\nu(\xi)^T A\,\nu(\xi) > 0.
		\] Hence, once $t > 1$ is chosen large, at $s=0$
		\[
		q_\xi(\xi) = w(\xi) - 1 < w(\xi).
		\] and at  $s=\delta_*$
		\[
		q_\xi(\xi + \delta_* \nu(\xi)) > w(\xi + \delta_* \nu(\xi)). 
		\] 
		Since  $q_\xi(\xi+s\nu(\xi))$ is continuous at $s=0$ and $\partial D$ is compact, then exist finitely  points $\xi_1,\xi_2, \cdots, \xi_N \in \partial D $, and $ \delta_0 \in(0,\delta_*)$ such that 
		\[
		\max_{1 \leq j \leq N} q_j \leq w - \frac{1}{2} \quad \text{on } \overline{U}_{\delta_0}, \qquad
		\max_{1 \leq j \leq N} q_j \geq w + \frac{1}{2} \quad \text{on } \{ \delta_* - \delta_0 \le d(x) \le \delta_* \}.
		\]
		Since $\overline{U_{\delta_*}} \Subset E_1$, we choose $T > 1$ and $\gamma \in \mathbb{R}$ so that
		\[
		u_{\mathrm{mid}} := Ts + \gamma
		\]
		satisfies
		\[
		u_{\mathrm{mid}} \le w - \frac{1}{2} \quad \text{on } \overline{U_{\delta_0}}, \qquad
		u_{\mathrm{mid}} \ge \max_{1 \le j \le N} q_j + \frac{1}{2} \quad \text{on } \partial E_1.
		\]
		Because $D^2 u_{\mathrm{mid}} = T A$, we also have
		\begin{equation}
			F(D^2 u_{\mathrm{mid}}) = \left( \frac{\sigma_k(TA)}{\sigma_l(TA)} \right)^{1/(k-l)}
			= T > 1,
		\end{equation}
		where the last equality follows from $A \in \mathcal{A}_{k,l}$ and the homogeneity of $\sigma_k$ and $\sigma_l$. 
		Finally, the function $u_{\mathrm{near}}$ is obtained by gluing $w$, $q_1,\dots,q_N$, and $u_{\mathrm{mid}}$ via the regularized maximum in Lemma \ref{lem:regularized_max}, we obtain
		\begin{equation}
			u_{\mathrm{near}} \in C^\infty(E_1 \setminus \overline{D}) \cap C^0(\overline{E_1} \setminus D),
		\end{equation}
		such that
		\[
		u_{\mathrm{near}} = w \quad \text{near } \partial D, \qquad
		u_{\mathrm{near}} = Ts + \gamma \quad \text{near } \partial E_1.
		\]
		Therefore $u_{\mathrm{near}} = \varphi$ on $\partial D$, and $u_{\mathrm{near}}$ is a strictly $\widetilde{\Gamma}_k$-admissible subsolution of the quotient equation
		\[
		\frac{\sigma_k(\eta(D^2 u))}{\sigma_l(\eta(D^2 u))} = 1,
		\]
		in $E_1 \setminus \overline{D}$. $\square$
	\end{proof}
	
	\subsection*{Far field correction}
	
	Let
	\[
	F^{ij}(A) := \frac{\partial F}{\partial M_{ij}}(A), \qquad L_A \psi := F^{ij}(A) \psi_{ij}.
	\]
	By Proposition \ref{prop}, the matrix $(F^{ij}(A))$ is positive definite. Let $v$ be the unique solution of
	\[
	\begin{cases}
		L_A v = 0 & \text{in } \mathbb{R}^n \setminus \overline{E_1}, \\
		v = 1 & \text{on } \partial E_1, \\
		v(x) \to 0 & \text{as } |x| \to \infty.
	\end{cases}
	\]
	Since $L_A$ is a linear, constant-coefficient, uniformly elliptic operator of second order, this exterior Dirichlet problem is, in the sense of potential theory, the capacitary potential of an ellipsoid. Hence, by the classical potential theory for ellipsoids, there exists a unique solution $v$, and it satisfies
	\[
	0 < v < 1, \quad |Dv| > 0 \quad \text{in } \mathbb{R}^n \setminus \overline{E_1},
	\]
	and, for $|x| \gg 1$ and some $C > 0$,
	\[
	C^{-1}|x|^{2-n} \le v(x) \le C|x|^{2-n}, \quad
	C^{-1}|x|^{1-n} \le |Dv(x)| \le C|x|^{1-n}, \quad
	|D^2 v(x)| \le C|x|^{-n}. 
	\]
	Choose $\theta \in (0,1)$ and set $W := v^\theta$. From the chain rule,
	\begin{equation}
		DW = \theta v^{\theta-1} Dv, \qquad
		D^2 W = \theta v^{\theta-1} D^2 v + \theta(\theta-1) v^{\theta-2} Dv \otimes Dv. 
	\end{equation}
	Since $L_A v = 0$, we have
	\begin{equation}
		L_A W = \theta(\theta-1) v^{\theta-2} F^{ij}(A) v_i v_j < 0 \quad \text{in } \mathbb{R}^n \setminus \overline{E_1}, 
	\end{equation}
	where the inequality follows from $\theta(\theta-1)<0$ and the positive definiteness of $(F^{ij}(A))$.
	
	Because $(F^{ij}(A))$ is positive definite, there exists $c_0 > 0$ such that
	\[
	F^{ij}(A)\xi_i\xi_j \ge c_0|\xi|^2 \qquad \text{for all } \xi \in \mathbb{R}^n.
	\]
	Hence
	\begin{equation}
		-L_A W \ge c_0 \theta(1-\theta) v^{\theta-2} |Dv|^2. 
	\end{equation}
	On the fixed annulus $\overline{E_{2\tau_1}\setminus E_1}$, the functions $v$ and $|Dv|$ are bounded below and $D^2 v$ is bounded. Therefore, $|D^2 W| \le C\theta$ and 
	\begin{align*}
		|D^2W|^2 
		&\le C\theta^2\left( v^{2\theta-2}|D^2v|^2 + v^{2\theta-4}|Dv|^4 \right) \\
		&\le C\theta^2(1 + |Dv|^2)  \\
		&\le C\theta^2\left(1 + \frac{1}{c_0(1-\theta)v^{\theta-2}}(-L_AW)\right) \\
		&\le C\theta(-L_AW) . 
	\end{align*}
	For $|x| \gg 1$, let $\beta := \theta(n-2) \in (0, n-2)$, and $|D^2 W| \le C |x|^{-2-\beta}$
	\begin{align*}
		|D^2W|^2 
		&\le C\theta^2\left( v^{2\theta-2}|D^2v|^2 + v^{2\theta-4}|Dv|^4 \right)\\
		&\le C\theta^2 |x|^{-4-2\beta} \\
		&\le C\theta |x|^{-2-\beta}  \\
		&\le C\theta(-L_AW) ,
	\end{align*}
	after enlarging $C$ if necessary, we have
	\begin{equation}
		|D^2 W|^2 \le C\theta (-L_A W) \qquad \text{in } \mathbb{R}^n \setminus E_1.
	\end{equation}
	Also, from $W=v^\theta$ ,
	\begin{equation}
		W(x) = O(|x|^{-\beta}) \qquad \text{as } |x| \to \infty. 
	\end{equation}
	Set
	\begin{equation}
		\kappa_0 := \min_{\partial E_1} (-\partial_\nu W) > 0, \qquad m_0 := \max_{\partial E_1} \partial_\nu s < \infty. 
	\end{equation}
	For $\tau > 1$, define
	\[
	E_\tau := \{x \in \mathbb{R}^n : s(x) < \tau \tau_1\}, \qquad W_\tau(x) := W(x/\sqrt{\tau}).
	\]
	Then $W_\tau = 1$ on $\partial E_\tau$, and by the chain rule,
	\[
	DW_\tau(x) = \tau^{-1/2} DW(x/\sqrt{\tau}), \qquad
	D^2 W_\tau(x) = \tau^{-1} D^2 W(x/\sqrt{\tau}).
	\]
	Therefore
	\begin{equation}
		L_A W_\tau = \tau^{-1} (L_A W)(x/\sqrt{\tau}) < 0, 
	\end{equation}
	and 
	\[
	|D^2 W_\tau| \le \frac{C\theta}{\tau},\qquad |D^2 W_\tau|^2 \le \frac{C\theta}{\tau} (-L_A W_\tau). 
	\]
	Moreover, the definition of $\kappa_0$,
	\begin{equation}
		-\partial_\nu W_\tau \ge \kappa_0 \tau^{-1/2} \qquad \text{on } \partial E_\tau,
	\end{equation}
	and
	\begin{equation}
		\partial_\nu s \le m_0 \tau^{1/2} \qquad \text{on } \partial E_\tau. 
	\end{equation}
	
	\begin{proposition}[Far field barrier]
		Let $u_{\mathrm{mid}}(x) = Ts(x) + \gamma$ be the function obtained in Proposition \ref{near}. 
		Then there exist $\beta \in (0, n-2)$ and $c_* > \gamma$ such that, for every $c > c_*$, 
		one can choose $\tau > 1$ and $m > 0$ so that
		\[
		u_{\mathrm{far}}(x) := Q_c(x) - mW_{\tau}(x) \quad \text{in } \mathbb{R}^n \setminus E_{\tau},
		\]
		and
		\[
		u_{\mathrm{far}} = u_{\mathrm{mid}} \quad \text{on } \partial E_{\tau}, \qquad
		\partial_{\nu} u_{\mathrm{far}} > \partial_{\nu} u_{\mathrm{mid}} \quad \text{on } \partial E_{\tau},
		\]
		with
		\[
		Q_c(x) - C|x|^{-\beta} \le u_{\mathrm{far}}(x) \le Q_c(x) \quad \text{for } |x| \gg 1.
		\]
		Moreover, $u_{\mathrm{far}}$ is a strictly $\widetilde{\Gamma}_k$-admissible subsolution of the mixed quotient equation in $\mathbb{R}^n \setminus E_{\tau}$.
	\end{proposition}
	
	\begin{proof}
		Since $F$ is the quotient operator
		\[
		F(M)=\left(\frac{\sigma_k(\eta(M))}{\sigma_l(\eta(M))}\right)^{1/(k-l)},
		\]
		it is positively homogeneous of degree one on the admissible cone $\widetilde{\Gamma}_k$. In particular, for $A\in\mathcal{A}_{k,l}$, we have $F(A)=1$. Moreover, $F$ is smooth in a neighborhood of $A$, so shrinking this neighborhood if necessary, there exist constants $\rho>0$ and $C_0>0$, depending only on $n,k,l,A$, such that whenever $|N|\le\rho$,
		\begin{equation}
			F(A+N)\ge 1+F^{ij}(A)N_{ij}-C_0|N|^2.
		\end{equation}
		
		Set $N:=-mD^2W_\tau$. By choosing $\theta>0$ sufficiently small (so that $|D^2W_\tau|\le C\theta/\tau\le \rho$), we have $|N|\le\rho$. 
		\[
		\begin{aligned}
			F(D^2u_{\mathrm{far}})
			&=F(A-mD^2W_\tau)\\
			&\ge 1 - mF^{ij}(A)(W_\tau)_{ij}-C_0m^2|D^2W_\tau|^2\\
			&=1+m(-L_AW_\tau)-C_0m^2|D^2W_\tau|^2.
		\end{aligned}
		\]
		Using the estimate $|D^2W_\tau|^2\le C\theta(-L_AW_\tau)/\tau$ and $\tau>1$,
		\begin{equation}
			F(D^2u_{\mathrm{far}})
			\ge 1+m(1-C_0C\theta)(-L_AW_\tau).
		\end{equation}
		Choose $\theta>0$ small so that $C_0C\theta\le 1/2$. Since $-L_AW_\tau>0$,
		\begin{equation}
			F(D^2u_{\mathrm{far}})\ge 1+\frac{m}{2}(-L_AW_\tau)>1.
		\end{equation}
		Thus $u_{\mathrm{far}}$ is a strict subsolution.
		
		It remains to verify the matching conditions. On $\partial E_\tau$, $W_\tau=1$ and $s=\tau\tau_1$. The condition $u_{\mathrm{far}}=u_{\mathrm{mid}}$ gives
		\begin{equation}
			c+\tau\tau_1-m=T\tau\tau_1+\gamma,
		\end{equation}
		hence
		\begin{equation}
			m=c-\gamma-(T-1)\tau\tau_1. 
		\end{equation}
		The normal derivative condition $\partial_\nu u_{\mathrm{far}}>\partial_\nu u_{\mathrm{mid}}$ is equivalent to
		\[
		m(-\partial_\nu W_\tau)>(T-1)\partial_\nu s.
		\]
		Using $-\partial_\nu W_\tau\ge \kappa_0\tau^{-1/2}$ and $\partial_\nu s\le m_0\tau^{1/2}$, it suffices that
		\[
		m>\frac{(T-1)m_0}{\kappa_0}\tau. 
		\]
		we need
		\begin{equation}
			c-\gamma-(T-1)\tau\tau_1>\frac{(T-1)m_0}{\kappa_0}\tau.
		\end{equation}
		Taking
		\begin{equation}
			c_*:=\gamma+(T-1)\tau_1+\frac{(T-1)m_0}{\kappa_0},
		\end{equation}
		then for any $c>c_*$, such $\tau$ and $m$ can be chosen. The far-field estimate
		\[
		Q_c(x)-C|x|^{-\beta}\le u_{\mathrm{far}}(x)\le Q_c(x)\qquad (|x|\gg1)
		\]
		follows from $W_\tau(x)=O(|x|^{-\beta})$. This completes the proof.
	\end{proof}

	\section{Approximating Dirichlet problems and a priori estimates}
	Define 
	$$	\underline{u}(x) :=
	\begin{cases}
		u_{\text{near}}(x), & x \in E_1 \setminus \overline{D}, \\[6pt]
		u_{\text{mid}}(x) = T s(x) + \gamma, & x \in E_\tau \setminus E_1, \\[6pt]
		u_{\text{far}}(x) = Q_c(x) - m W_\tau(x), & x \in \mathbb{R}^n \setminus E_\tau.
	\end{cases}$$
	Let $\Omega_R = E_R \setminus \overline{D}$. Consider the approximating problem
	\[
	\begin{cases}
		F(D^2 u_R) = 1, & \text{in } \Omega_R, \\
		u_R = \varphi, & \text{on } \partial D, \\
		u_R = \bar{u}_R, & \text{on } \partial E_R,
	\end{cases}
	\]
	where $\bar{u}_R(x) = Q_c(x) + \bar{C} R^{-\beta/2}$ and choose $\bar{C}$  so large that for all sufficiently large $R$, we have $\bar{u}_R \ge \varphi$ on $\partial D$ and 
	$\bar{u}_R > \underline{u}$ on $\partial E_R$.
	To apply Guan's theorem, we need a smooth strict subsolution. 
	First, applying Lemma \ref{lem:regularized_max} in a thin annulus around $\partial E_\tau$, 
	we glue $u_{\mathrm{near}}$, $u_{\mathrm{mid}}$, and $u_{\mathrm{far}}$ into a smooth strict subsolution 
	$v \in C^\infty(E_{2\tau}\setminus\overline{D})$. 
	Then, in a collar of $\partial E_R$, let
	\[
	q_R(x):=\Lambda s(x)-\left(\Lambda R-R-c-\bar{C}R^{-\beta/2}\right),
	\]
	where $\Lambda>1$ is chosen so that $q_R<\varphi$ on $\partial D$ and $q_R=\bar{u}_R$ on $\partial E_R$. 
	Since $F(D^2q_R)=\Lambda>1$, $q_R$ is a strict subsolution. 
	Applying Lemma \ref{lem:regularized_max} to $v$ and $q_R(x)$  in a collar of  $\partial E_R$ obtain 
	\[
	w_R\in C^\infty(\overline{\Omega_R})
	\]
	such that
	\[
	w_R=\varphi \text{ on } \partial D, \qquad w_R=\bar{u}_R \text{ on } \partial E_R, \qquad F(D^2w_R)>1 \text{ in } \Omega_R.
	\]
	Then Guan's theorem applies to $w_R$, yielding a unique smooth admissible solution $u_R$.
	By Guan's theorem \cite{Guan2023} applied to the bounded domain $\Omega_R$, there exists a unique smooth $\widetilde{\Gamma}_k$-admissible solution $u_R \in C^\infty(\Omega_R) \cap C^0(\overline{\Omega}_R)$.
	
	\subsection{\(C^0\) and \(C^1\) estimates}
	
	\begin{proposition}
		There exists $C > 0$, independent of $R$, such that
		\begin{itemize}
			\item[(i)] $\underline{u} \le u_R \le \bar{u}_R$ in $\Omega_R$;
			\item[(ii)] $|Du_R| \le C$ on $\partial D$;
			\item[(iii)] $|Du_R| \le C R^{1/2}$ on $\partial E_R$;
			\item[(iv)] $\displaystyle \max_{\overline{\Omega_R}} |Du_R| = \max_{\partial \Omega_R} |Du_R|$;
			\item[(v)] if $k=n$, then $\partial_\nu u_R \ge c_1 > 0$ on $\partial D$.
		\end{itemize}
	\end{proposition}
	
	\begin{proof}
		We prove the five estimates in order.
		
		\paragraph{Proof of (i).} 
		The three functions satisfy
		\[
		F(D^2 \underline{u}) \ge 1 = F(D^2 u_R) = F(D^2 \bar{u}_R) \quad \text{in } \Omega_R.
		\]
		On the boundary $\partial \Omega_R = \partial D \cup \partial E_R$, we have
		\begin{equation}
			\underline{u} = u_R \le \bar{u}_R \quad \text{on } \partial D,
		\end{equation}
		and
		\begin{equation}
			\underline{u} \le \bar{u}_R = u_R \quad \text{on } \partial E_R.
		\end{equation}
		Thus $\underline{u} \le u_R \le \bar{u}_R$ on $\partial \Omega_R$. By the comparison principle for the quotient equation, we obtain
		\begin{equation}
			\underline{u} \le u_R \le \bar{u}_R \quad \text{in } \Omega_R.
		\end{equation}
		
		\paragraph{Proof of (ii).}
		Since $u_R$ is $\widetilde{\Gamma}_k$-admissible, we have $\eta(D^2 u_R) \in \Gamma_k \subset \Gamma_1$, hence
		\begin{equation}
			(n-1)\Delta u_R = \sigma_1(\eta(D^2 u_R)) > 0 \quad \text{in } \Omega_R.
		\end{equation}
		Let $\Phi$ be a smooth extension of $\varphi$ to $E_1 \setminus \overline{D}$ with $\partial_\nu \Phi = 0$ on $\partial D$, and set $w_R = u_R - \Phi$. Then $w_R = 0$ on $\partial D$, and by (1), $w_R$ is uniformly bounded on $\partial E_1$.
		
		From the above, $\Delta w_R = \Delta u_R - \Delta \Phi \ge -C_3$. Let $\psi \in C^\infty(E_1 \setminus \overline{D})$ solve
		\[
		\Delta \psi = -1 \quad \text{in } E_1 \setminus \overline{D}, \qquad
		\psi = 0 \text{ on } \partial D, \quad \psi = 1 \text{ on } \partial E_1.
		\]
		By the maximum principle and Hopf's lemma,
		\[
		0 < \psi < 1 \quad \text{in } E_1 \setminus \overline{D}, \qquad
		\partial_\nu \psi \ge c_2 > 0 \quad \text{on } \partial D.
		\]
		By the maximum principle, $w_R \le C_3 \psi$ in $E_1 \setminus \overline{D}$. Consequently,
		\[
		u_{\mathrm{near}} \le u_R \le \Phi + C_3 \psi \quad \text{in } E_1 \setminus \overline{D}.
		\]
		All three functions agree with $\varphi$ on $\partial D$, so taking normal derivatives gives
		\begin{equation}
			|\partial_\nu u_R| \le C \quad \text{on } \partial D.
		\end{equation}
		Tangential derivatives are bounded by $|\nabla_\tau u_R| = |\nabla_\tau \varphi| \le C$. Hence $|Du_R| \le C$ on $\partial D$.
		
		\paragraph{Proof of(iii).}
		From the $C^0$ estimate,
		\begin{equation}
			q_R \le u_R \le \bar{u}_R \quad \text{in } \Omega_R,
		\end{equation}
		and all three functions agree on $\partial E_R$. By Hopf's lemma,
		\begin{equation}
			\partial_{\nu_R} q_R \le \partial_{\nu_R} u_R \le \partial_{\nu_R} \bar{u}_R \quad \text{on } \partial E_R.
		\end{equation}
		Since $D q_R = \Lambda D s$ and $D \bar{u}_R = D s$, and $|Ds| \le C R^{1/2}$ on $\partial E_R$, we get
		\begin{equation}
			|\partial_{\nu_R} u_R| \le C R^{1/2}.
		\end{equation}
		As $u_R$ is constant on $\partial E_R$, all tangential derivatives vanish. Therefore
		\begin{equation}
			|Du_R| \le C R^{1/2} \quad \text{on } \partial E_R.
		\end{equation}
		
		\paragraph{Proof of (iv).}
		Differentiating $F(D^2 u_R) = 1$ with respect to $x_l$ gives
		\[
		F^{ij}(D^2 u_R) (u_R)_{ijl} = 0.
		\]
		Then
		\[
		F^{ij}(D^2 u_R) (|Du_R|^2)_{ij} = 2 F^{ij}(D^2 u_R) (u_R)_{il} (u_R)_{jl} \ge 0.
		\]
		By the maximum principle, $|Du_R|^2$ attains its maximum on $\partial \Omega_R$, hence the result.
		\paragraph{Proof of (v).}
		If $k=n$, then near $\partial D$ the lower barrier is
		\[
		\varphi(\pi(x))+Kd(x)+\frac{B}{2}d(x)^2.
		\]
		Since $u_R\ge \underline{u}$ and both functions equal $\varphi$ on $\partial D$, Hopf's lemma gives
		\begin{equation}
			\partial_\nu u_R \ge K \quad \text{on } \partial D.
		\end{equation}
		This proves (5) with $c_1:=K$.
	\end{proof}
	
	\subsection{Boundary second derivative estimate}
	
	\begin{proposition}
		There exists $C > 0$, independent of $R$, such that
		\[
		|D^2 u_R| \le C \quad \text{on } \partial \Omega_R.
		\]
	\end{proposition}
	
	\begin{proof}
		We estimate the second derivatives on $\partial D$ and on $\partial E_R$ separately.
		
		\paragraph{Inner boundary $\partial D$.}
		Fix $p \in \partial D$. Choose local coordinates $y = (y', y_n)$ centered at $p$ such that
		\[
		(\mathbb{R}^n \setminus \overline{D}) \cap B_r(0) = \{ y_n > \rho(y') \} \cap B_r(0), \qquad \rho(0) = 0, \quad D\rho(0) = 0.
		\]
		Write $u := u_R$. Since $u = \Phi$ on $\partial D$,
		\[
		u(y', \rho(y')) = \Phi(y', \rho(y')) \quad \text{for } |y'| \ll 1.
		\]
		Differentiating once gives
		\[
		u_\alpha + u_n \rho_\alpha = \Phi_\alpha + \Phi_n \rho_\alpha.
		\]
		At $y=0$, because $D\rho(0)=0$, we get $u_\alpha(0) = \Phi_\alpha(0)$. Differentiating again yields
		\[
		u_{\alpha\beta}(0) = \Phi_{\alpha\beta}(0) - u_n(0)\rho_{\alpha\beta}(0).
		\]
		Since $|u_n| \le C$, we obtain $|u_{\alpha\beta}(0)| \le C$.

		For the mixed derivatives.
		Define the linearized operator
		\begin{equation}
			\mathcal{L} := F^{ij}(D^2 u_R)\partial_{ij}. \label{eq:L_mixed}
		\end{equation}
		Since $F(D^2 u_R)=1$ is constant, differentiating the equation gives
		\begin{equation}
			\mathcal{L}(\partial_\alpha u_R)=0 \qquad (\alpha=1,\dots,n). \label{eq:Lder_mixed}
		\end{equation}
		In local coordinates at $p$, we have
		$$D_\delta := \{ x \in \mathbb{R}^n \setminus \overline{D} : |x| < \delta \}.$$
		We first prove a key estimate for a barrier function. 
		Since $w_R$ is a strict subsolution and $u_R$ is a solution of the approximating problem, 
		the comparison principle gives $u_R-w_R>0$ in $\Omega_R$, and Hopf's lemma yields 
		$\partial_\nu(u_R-w_R)>0$ on $\partial D$. 
		We have
		\begin{equation}
			0<F(D^2w_R)-F(D^2u_R)
			\le F^{ij}(D^2u_R)(w_R-u_R)_{ij}
			= \mathcal{L}(w_R-u_R). \label{eq:concavity_zeta_mixed}
		\end{equation}
		Since $\mathcal{L}u_R=0$, this gives $\mathcal{L}w_R>0$, hence
		\begin{equation}
			\mathcal{L}(u_R-w_R)=-\mathcal{L}w_R<0 \quad \text{in } D_\delta. \label{eq:Luw_zeta_mixed}
		\end{equation}
		
		Now set
		\[
		\zeta:=u_R-w_R+td-\frac{N}{2}d^2.
		\]
		Using \eqref{eq:Luw_zeta_mixed}, the smoothness of $d$, 
		$$
		\begin{aligned}
			\mathcal L\left(\frac{d^2}{2}\right)
			&= F^{ij}\partial_i d\,\partial_j d + d\,\mathcal L d \\
			&\ge \lambda_0\sum_i F^{ii} - d\,C_d\sum_i F^{ii} \\
			&= (\lambda_0 - C C_d)\sum_i F^{ii}.
		\end{aligned}$$
		Since $0<d<\delta$ in $D_\delta$, taking $\delta>0$ sufficiently small so that $\lambda_0 - C_d\delta \ge \lambda_0/2$, we obtain
		\begin{equation}
			\mathcal L\left(\frac{d^2}{2}\right) \ge \frac{\lambda_0}{2}\sum_i F^{ii}.
		\end{equation}
		we obtain, for $t>0$ sufficiently small and $N>0$ sufficiently large,
		\begin{equation}
			\begin{aligned}
				\mathcal{L}\zeta
				&= \mathcal{L}(u_R-w_R) + t\mathcal{L}d - N\mathcal{L}\left(\frac{d^2}{2}\right) \\
				&\le tC_d\sum_i F^{ii} - \frac{N\lambda_0}{2}\sum_i F^{ii} \\
				&= \left(tC_d-\frac{N\lambda_0}{2}\right)\sum_i F^{ii} \\
				&\le  -\sum_iF^{ii},
			\end{aligned}
			\label{eq:Lzeta_mixed}
		\end{equation}
		for $tC_d-\frac{N\lambda_0}{2}\leq -1$ and $C_d = n\max_{\partial D}|\kappa_\alpha|$. 
		Moreover, by Hopf's lemma, $u_R-w_R\ge cd$ near $\partial D$. Taking $\delta>0$ sufficiently small, we have $\zeta\ge0$ on $\partial D_\delta$.
		
		Now define the auxiliary function
		\[
		\Psi := A_1\zeta + A_2 d - A_3\sum_{\beta=1}^{n-1}|\partial_\beta(u_R-w_R)|^2,
		\]
		where $A_1,A_2,A_3>0$ are constants to be chosen with $A_1\gg A_2\gg A_3\gg 1$.
		
		We estimate $\mathcal{L}\Psi$ term by term. From \eqref{eq:Lzeta_mixed},
		\[
		\mathcal{L}(A_1\zeta) \le -A_1\sum_i F^{ii}.
		\]
		Since $d$ is smooth, there exists $C_d>0$ such that
		\[
		\mathcal{L}(A_2d) \le A_2 C_d \sum_i F^{ii}.
		\]
		For the last term, set
		\[
		v_\beta:=\partial_\beta(u_R-w_R) \qquad (\beta=1,\dots,n-1).
		\]
		From \eqref{eq:Lder_mixed}, $\mathcal{L}v_\beta = -\mathcal{L}(\partial_\beta w_R)$. 
		Since $w_R=u_{\mathrm{near}}$ near $\partial D$ is fixed and smooth, there exists $C_2>0$ such that
		\[
		|\mathcal{L}v_\beta| \le C_2 \sum_i F^{ii}.
		\]
		By the $C^1$ estimate, $|v_\beta|\le C$, hence
		\[
		|2v_\beta \mathcal{L}v_\beta| \le C_3 \sum_i F^{ii}.
		\]
		By the ellipticity estimate in Proposition \ref{prop}, there exists $c_0>0$ such that
		\[
		\sum_{\beta<n} F^{ij}\partial_i v_\beta \partial_j v_\beta \ge c_0 \sum_i F^{ii}.
		\]
		Therefore,
		\[
		\mathcal{L}\left(\sum_{\beta<n}v_\beta^2\right)
		= 2\sum_{\beta<n}F^{ij}\partial_i v_\beta \partial_j v_\beta
		+ 2\sum_{\beta<n}v_\beta \mathcal{L}v_\beta
		\ge (2c_0 - C_3)\sum_i F^{ii}.
		\]
		Consequently,
		\[
		\mathcal{L}\left(-A_3\sum_{\beta<n}v_\beta^2\right)
		\le -A_3(2c_0 - C_3)\sum_i F^{ii}.
		\]
		Combining the above estimates, we obtain
		\begin{align}
			\mathcal{L}\Psi
			&\le -A_1 \sum_iF^{ii}
			+ A_2C_d\sum_i F^{ii}
			- A_3(2c_0 - C_3)\sum_i F^{ii} \notag \\
			&= -\left(A_1- A_2C_d + A_3(2c_0 - C_3)\right)\sum_i F^{ii}. \label{eq:LPsi_est_mixed}
		\end{align}
		Choosing $A_1\gg A_2+A_3$ so that
		\[
		A_1 - A_2C_d + A_3(2c_0 - C_3) \ge \frac{A_1}{2},
		\]
		we get
		\begin{equation}
			\mathcal{L}\Psi \le -\frac{A_1}{2}\sum_i F^{ii} < 0 \quad \text{in } D_\delta. \label{eq:LPsi_mixed}
		\end{equation}
		On $\partial D\cap B_\delta$, we have $d=0$, $u_R=w_R$, and $\partial_\beta(u_R-w_R)=0$, so $\Psi=0$. 
		On $\partial B_\delta\cap D_\delta$, by Hopf's lemma, $u_R-w_R\ge cd$ near $\partial D$. Hence
		\[
		\zeta = u_R-w_R+td-\frac{N}{2}d^2 \ge cd - \frac{N}{2}d^2 \ge d\left(c-\frac{N}{2}\delta\right)>0
		\]
		by taking $\delta>0$ sufficiently small. Thus $\Psi=A_1\zeta+A_2d-A_3\sum|\partial_\beta v|^2\ge0$ for $A_1$ sufficiently large. 
		Therefore $\Psi\ge0$ on $\partial D_\delta$. By the maximum principle, together with \eqref{eq:LPsi_mixed},
		\begin{equation}
			\Psi \ge 0 \quad \text{in } D_\delta. \label{eq:Psi_bound_mixed}
		\end{equation}
		Now fix $\alpha\in\{1,\dots,n-1\}$. For either choice of sign $\pm$, define
		\[
		\Psi_\pm := \Psi \pm \partial_\alpha(u_R-w_R).
		\]
		We verify that $\Psi_\pm$ satisfies the hypotheses of the maximum principle.
		Indeed, by \eqref{eq:Lder_mixed} and \eqref{eq:LPsi_mixed},
		\[
		\mathcal{L}(\Psi_\pm)
		= \mathcal{L}\Psi \pm \mathcal{L}(\partial_\alpha u_R) \mp \mathcal{L}(\partial_\alpha w_R)
		= \mathcal{L}\Psi \mp \mathcal{L}(\partial_\alpha w_R) < 0 \quad \text{in } D_\delta,
		\]
		after possibly increasing $A_1$ to absorb the term $\mathcal{L}(\partial_\alpha w_R)$.
		On the boundary $\partial D_\delta$, by \eqref{eq:Psi_bound_mixed} and the $C^1$ bound,
		\[
		\Psi_\pm = \Psi \pm \partial_\alpha(u_R-w_R) \ge 0 \quad \text{on } \partial D_\delta.
		\]
		Therefore, by the maximum principle,
		\begin{equation}
			\Psi \pm \partial_\alpha(u_R-w_R) \ge 0 \quad \text{in } D_\delta. \label{eq:Psi_alpha_bound_mixed}
		\end{equation}
		At $p\in\partial D$, we have $\Psi(0)=0$ and $\partial_\alpha(u_R-w_R)(0)=0$. 
		Applying Hopf's lemma to $\Psi_+$ and $\Psi_-$ yields
		\[
		\partial_n\Psi(0) \pm \partial_n\partial_\alpha(u_R-w_R)(0) \ge 0.
		\]
		Thus
		\[
		|u_{\alpha n}(0) - w_{\alpha n}(0)| \le \partial_n\Psi(0).
		\]
		Since $w_R=u_{\mathrm{near}}$ near $\partial D$ is fixed and smooth, and $\Psi$ is composed of $C^1$-bounded functions, we have
		\[
		|u_{\alpha n}(0)| \le |w_{\alpha n}(0)| + |\partial_n\Psi(0)| \le C,
		\]
		where $C$ is independent of $R$. Since $p\in\partial D$ and $\alpha\in\{1,\dots,n-1\}$ are arbitrary,
		\begin{equation}
			|u_{\alpha n}| \le C \quad \text{on } \partial D. \label{eq:mixed_mixed}
		\end{equation}

		We now estimate $u_{nn}(0)$ for $k\leq n-1$. By the boundedness of tangential and mixed derivatives,
		\[
		\Delta u \, I_n - D^2u = 
		\begin{pmatrix}
			\Delta u - u_{11} & -u_{12} & \cdots & -u_{1n} \\
			-u_{21} & \Delta u - u_{22} & \cdots & -u_{2n} \\
			\vdots & \vdots & \ddots & \vdots \\
			-u_{n1} & -u_{n2} & \cdots & \Delta u - u_{nn}
		\end{pmatrix}.
		\]
		It follows that 
		\[
		\eta(D^2 u)(0) = \begin{pmatrix}
			u_{nn} I_{n-1} + S & \xi \\
			\xi^T & \zeta
		\end{pmatrix},
		\qquad |S| + |\xi| + |\zeta| \le C.
		\]
		Admissibility implies $D^2 u(0) \in \widetilde{\Gamma}_1$, so $\Delta u(0) > 0$. Since tangential derivatives are bounded, this gives $u_{nn} \ge -C$.
		
		For the quotient equation, $\sigma_k(\eta) = \sigma_l(\eta)$. For large $u_{nn}$, the leading terms give
		\[
		C_{n-1}^k  u_{nn}^k - C u_{nn}^{k-1} \le C_{n-1}^l  u_{nn}^l + C u_{nn}^{l-1}.
		\]
		Since $k > l$, this yields $u_{nn} \le C$. Hence $|u_{nn}(0)| \le C$. Therefore, $|D^2 u(0)| \le C$.
		
		For $k=n$, set
		\[
		T := \sum_{\alpha < n} u_{\alpha\alpha}(0).
		\]
		Since $\Phi_n(0) = 0$, summing  over $\alpha = 1, \dots, n-1$ gives
		\[
		T = \sum_{\alpha < n} \Phi_{\alpha\alpha}(0) - u_n(0) \sum_{\alpha < n} \rho_{\alpha\alpha}(0).
		\]
		The quantity $-\sum_{\alpha < n} \rho_{\alpha\alpha}(0)$ is the mean curvature of $\partial D$ at $p$ with respect to the outward unit normal of $D$. By strict mean convexity and the lower bound $D_\nu u_R \ge K$ , we obtain
		\[
		0 < c_3 \le T \le C, \qquad c_3 = c_3(\partial D, K) > 0.
		\]
		Also,
		\[
		\eta(D^2 u)(0) = \begin{pmatrix}
			u_{nn} I_{n-1} + S & \xi \\
			\xi^T & T
		\end{pmatrix},
		\qquad |S| + |\xi|+|T| \le C.
		\]
		Admissibility gives $\eta(D^2 u)(0) > 0$, so in particular $u_{nn} \ge -C$. Since $F(D^2 u) = 1$ and $k=n$, we have $\sigma_n(\eta(D^2 u)(0)) = \sigma_l(\eta(D^2 u)(0))$, i.e.
		\[
		\det(\eta(D^2 u)(0)) = \sigma_l(\eta(D^2 u)(0)).
		\]
		\noindent\textbf{Case 1: $l \le n-2$.}
		For the quotient equation, $\sigma_n(\eta) = \sigma_l(\eta)$ with $k=n$.
		For large $u_{nn}$, the block form gives
		\[
		\det(\eta) \ge T (u_{nn} - C)^{n-1},
		\]
		and
		\[
		\sigma_l(\eta) \le C_{n-1}^l u_{nn}^l + O(u_{nn}^{l-1}).
		\]
		Thus
		\[
		T (u_{nn} - C)^{n-1} \le C_{n-1}^l u_{nn}^l + O(u_{nn}^{l-1}).
		\]
		Since $l \le n-2$ and $T \ge c_3 > 0$, the left-hand side is of order $u_{nn}^{n-1}$ while the right-hand side is of order $u_{nn}^l$ with $l < n-1$. This is impossible for $u_{nn}$ sufficiently large. Hence $u_{nn} \le C$.

		\noindent\textbf{Case 2: $l=n-1$.}
		Using the block form above,
		\[
		\det(\eta(D^2 u)(0)) = T \det\left( u_{nn} I_{n-1} + S - T^{-1} \xi \xi^T \right).
		\]
		Since $T$ is bounded above and below by positive constants, and $S - T^{-1}\xi\xi^T$ is uniformly bounded, the left-hand side grows like
		Set $M:=S-T^{-1}\xi\xi^T$. Then
		\[
		\det(\eta)=T\det(u_{nn}I_{n-1}+M).
		\]
		Expanding the determinant,
		\begin{equation}
			\det(u_{nn} I_{n-1} + M)
			=
			u_{nn}^{n-1}
			+
			\sigma_1(M) u_{nn}^{n-2}
			+
			\sigma_2(M) u_{nn}^{n-3}
			+
			\cdots
			+
			\sigma_{n-1}(M).
		\end{equation}
		Hence
		\begin{equation}
			\det(\eta)
			=
			T u_{nn}^{n-1}
			+
			T\sigma_1(M) u_{nn}^{n-2}
			+
			T\sigma_2(M) u_{nn}^{n-3}
			+
			\cdots
			+
			T\sigma_{n-1}(M).
		\end{equation}
		On the other hand,
		\begin{equation}
			\sigma_{n-1}(\eta)
			=
			u_{nn}^{n-1}
			+
			\big((n-1)T+\operatorname{tr}(S)\big)u_{nn}^{n-2}
			+
			O(u_{nn}^{n-3}).
		\end{equation}
		Suppose, for contradiction, that $u_{nn}$ is unbounded on $\partial D$. 
		Then there exists a sequence of points $p_j \in \partial D$ such that
		\begin{equation}
			u_{nn}(p_j) \to \infty.
		\end{equation}
		By compactness of $\partial D$, passing to a subsequence if necessary, 
		we may assume $p_j \to p_0$ for some $p_0 \in \partial D$.
		
		For each $j$, the block decomposition at $p_j$ gives
		\[
		\eta(D^2u)(p_j)=\begin{pmatrix}
			u_{nn}(p_j)I_{n-1}+S_j & \xi_j\\
			\xi_j^T & T_j
		\end{pmatrix},
		\]
		where $T_j:=\sum_{\alpha<n}u_{\alpha\alpha}(p_j)$ satisfies $0<c_3\le T_j\le C$, and $|S_j|+|\xi_j|\le C$ by the tangential derivative estimates.
		
		From the expansions
		\begin{equation}
			\det(\eta_j)
			=
			T_j u_{nn}^{n-1}(p_j)
			+
			\bigl(T_j\operatorname{tr}(S_j)-|\xi_j|^2\bigr)u_{nn}^{n-2}(p_j)
			+
			O(u_{nn}^{n-3}(p_j)).
		\end{equation}
		and
		\begin{equation}
			\sigma_{n-1}(\eta_j)
			=
			u_{nn}^{n-1}(p_j)
			+
			\bigl((n-1)T_j+\operatorname{tr}(S_j)\bigr)u_{nn}^{n-2}(p_j)
			+
			O(u_{nn}^{n-3}(p_j)).
		\end{equation}
		together with $\det(\eta_j)=\sigma_{n-1}(\eta_j)$, comparing the $u_{nn}^{n-1}(p_j)$-terms yields $T_j\to1$. Then comparing the $u_{nn}^{n-2}(p_j)$-terms yields
		\[
		\operatorname{tr}(S_j)-|\xi_j|^2 = (n-1)T_j+\operatorname{tr}(S_j)+o(1).
		\]
		Since $T_j\to1$ and $S_j,\xi_j$ are bounded, passing to a subsequence if necessary, we may assume $S_j\to S_0$ and $\xi_j\to\xi_0$. Taking the limit gives
		\[
		\operatorname{tr}(S_0)-|\xi_0|^2 = (n-1)+\operatorname{tr}(S_0),
		\]
		which implies
		\begin{equation}
			-|\xi_0|^2 = n-1.
		\end{equation}
		This is impossible since the left-hand side is non-positive while the right-hand side is positive for $n\ge3$. Hence $u_{nn}$ cannot tend to infinity, so $u_{nn}\le C$ on $\partial D$. Therefore $|D^2u_R|\le C$ on $\partial D$.
		\paragraph{Outer boundary $\partial E_R$.}
		We rescale the outer boundary to a fixed ellipsoid. Set
		\[
		\widehat{E} := \{ y \in \mathbb{R}^n : s(y) < 1 \}, \qquad y := R^{-1/2} x,
		\qquad \widetilde{u}_R(y) := R^{-1} u_R(\sqrt{R} y), \qquad \widehat{\Omega}_R := \widehat{E} \setminus R^{-1/2} \overline{D}.
		\]
		Then $D_y^2 \widetilde{u}_R(y) = D_x^2 u_R(\sqrt{R} y)$, so $\widetilde{u}_R$ satisfies the same quotient equation
		\[
		F(D_y^2 \widetilde{u}_R) = 1 \quad \text{in } \widehat{\Omega}_R.
		\]
		The rescaled comparison functions satisfy $\widetilde{q}_R \le \widetilde{u}_R \le \widetilde{\bar{u}}_R$ and all three are constant on $\partial \widehat{E}$. By the same argument as on the inner boundary, we obtain
		$|D^2 \widetilde{u}_R|\le C$ on $ \partial \widehat{E}.$
		Scaling back,we obtain
		\begin{equation}
			|D^2 u_R| \le C \quad \text{on } \partial E_R.
		\end{equation}
	\end{proof}
	
	\subsection{Global \(C^2\) estimate}
	
	\begin{proposition}
		There exists $C > 0$, independent of $R$, such that
		\[
		|D^2 u_R| \le C \quad \text{in } \overline{\Omega_R}.
		\]
	\end{proposition}
	
	\begin{proof}
		Let $L_R := F^{ij}(D^2 u_R) \partial_{ij}$. Since $F(D^2 u_R) = 1$, differentiating twice and using the concavity of $F$ gives
		\begin{equation}
			L_R(\Delta u_R) \ge 0 \quad \text{in } \Omega_R.
		\end{equation}
		By the maximum principle, $\Delta u_R$ attains its maximum on $\partial \Omega_R$. Since $|D^2 u_R| \le C$ on $\partial \Omega_R$, we have $\Delta u_R \le C$ in $\Omega_R$.
		
		Since $u_R$ is $\widetilde{\Gamma}_k$-admissible, $\eta(D^2 u_R) \in \Gamma_k \subset \Gamma_2$, so $\sigma_2(\eta(D^2 u_R)) \ge 0$. Using the identity
		\[
		|\eta|^2 = \sigma_1(\eta)^2 - 2\sigma_2(\eta),
		\]
		we obtain
		\[
		|\eta(D^2 u_R)|^2 \le \sigma_1(\eta(D^2 u_R))^2.
		\]
		Since $\sigma_1(\eta(D^2 u_R)) = (n-1)\Delta u_R \le C$, we get $|\eta(D^2 u_R)| \le C$. Finally, since $D^2 u_R = \Delta u_R I - \eta(D^2 u_R)$, we conclude that $|D^2 u_R| \le C$ in $\overline{\Omega_R}$.
	\end{proof}
	
	\subsection{Passage to the limit}
	
	We now let $R \to \infty$. First, we show that the sequence $\{u_R\}$ is monotone. For $R_2 > R_1$, on $\partial D$ we have $u_{R_2} = u_{R_1} = \varphi$; on $\partial E_{R_1}$, since $\bar{u}_R$ is decreasing in $R$,
	\[
	u_{R_2} \le \bar{u}_{R_2} < \bar{u}_{R_1} = u_{R_1}.
	\]
	Thus $u_{R_2} \le u_{R_1}$ on $\partial \Omega_{R_1}$. By the comparison principle, $u_{R_2} \le u_{R_1}$ in $\Omega_{R_1}$. Hence $\{u_R\}$ is decreasing as $R$ increases.
	
	Fix a compact set $K \Subset \mathbb{R}^n \setminus \overline{D}$. By the global $C^2$ estimate, $\|u_R\|_{C^2(K)} \le C$ for all sufficiently large $R$. By the Arzel\`a--Ascoli theorem, there exists a subsequence converging in $C^2(K)$ to a limit $u$. Since the whole sequence $\{u_R\}$ is monotone, the entire sequence converges to the same limit.
	
	Let
	\[
	u(x) := \lim_{R \to \infty} u_R(x), \qquad x \in \mathbb{R}^n \setminus \overline{D}.
	\]
	Then $u \in C^2(\mathbb{R}^n \setminus \overline{D})$. Moreover, by the Evans--Krylov theorem and standard Schauder estimates, $u \in C^\infty(\mathbb{R}^n \setminus \overline{D})$.
	Passing to the limit in the equation gives
	\[
	F(D^2 u) = 1 \quad \text{in } \mathbb{R}^n \setminus \overline{D}.
	\]
	Passing to the limit in the boundary condition gives
	\[
	u = \varphi \quad \text{on } \partial D.
	\]
	Finally, taking $R \to \infty$ in the $C^0$ estimate $\underline{u} \le u_R \le \bar{u}_R$, we obtain
	\begin{equation}
		\underline{u}(x) \le u(x) \le Q_c(x) \quad \text{for } |x| \gg 1.
	\end{equation}
	Together with the far-field behavior of $\underline{u}$, this yields
	\begin{equation}
		Q_c(x) - C|x|^{-\beta} \le u(x) \le Q_c(x) \quad \text{for } |x| \gg 1.
	\end{equation}
	
	\section{Asymptotic expansion and uniqueness}
	\subsection{Asymptotic expansion}
	We recall the following lemma of Li--Li--Yuan \cite{LLY20}:
	
	\begin{lemma}\label{exter}
		Let $u \in C^\infty(\mathbb{R}^n\setminus \overline{D})$ be a smooth admissible solution of $F(D^2u)=1$ in an exterior domain, where $F$ is uniformly elliptic and concave on the admissible cone. Suppose
		\[
		\|D^2u\|_{L^\infty(\mathbb{R}^n\setminus \overline{D})} \le C.
		\]
		Then there exists a unique quadratic polynomial $Q$ such that
		\[
		|D^m(u-Q)(x)| \le C_m |x|^{2-n-m} \quad \text{as } |x| \to \infty,
		\]
		for every $m\ge 0$.
	\end{lemma}

	By the global $C^2$ estimate, $\|D^2 u\|_{L^\infty(\mathbb{R}^n \setminus \overline{D})} \le C$. 
	Moreover, by the Proposition\eqref{prop}, the quotient operator $F$ is uniformly elliptic 
	and concave on the admissible cone. Hence the Lemma\eqref{exter} applies to $u$, yielding a quadratic polynomial $Q$ such that
	\[
	|D^m(u-Q)(x)| \le C_m |x|^{2-n-m} \quad \text{as } |x| \to \infty,
	\]
	for every $m \ge 0$. In particular, for $m=0$,
	\[
	|u(x)-Q(x)| \le C_0 |x|^{2-n} \quad \text{as } |x| \to \infty,
	\]
	and for $m\ge 1$,
	\[
	|D^m(u-Q)(x)| \le C_m |x|^{2-n-m} \quad \text{as } |x| \to \infty.
	\]
	Since $u - Q_c = o(1)$ as $|x| \to \infty$, we must have $Q \equiv Q_c$. Hence
	\begin{equation}
		|D^m(u-Q_c)(x)| \le C_m |x|^{2-n-m} \quad \text{as } |x| \to \infty,
	\end{equation}
	for every $m \ge 0$.
	\subsection{Uniqueness}
	Let $u, v$ be two $\widetilde{\Gamma}_k$-admissible solutions of \eqref{eq:main} with the same boundary data $\varphi$ on $\partial D$. Set $w:=u-v$. Then $w=0$ on $\partial D$ and $w(x)=o(1)$ as $|x|\to\infty$.
	
	By the mean value theorem, for any bounded domain $\Omega\subset\mathbb{R}^n\setminus\overline{D}$,
	\[
	a^{ij}(x)w_{ij}=0 \quad \text{in } \Omega,
	\]
	where
	\[
	a^{ij}(x):=\int_0^1 F^{ij}(D^2u+t(D^2v-D^2u))\,dt.
	\]
	We claim that $a^{ij}$ is uniformly elliptic on $\Omega_R\setminus\overline{U_\delta}$. Indeed, for $2\le k\le n-1$, since $\eta$ is linear and $\Gamma_k$ is convex,
	\[
	\eta(tD^2u+(1-t)D^2v)
	=
	t\eta(D^2u)+(1-t)\eta(D^2v)
	\in \Gamma_k,\qquad 0\le t\le 1.
	\]
	For $k=n$, since $\eta(D^2u)>0$, $\eta(D^2v)>0$, and the positive definite cone is convex,
	\[
	\eta(tD^2u+(1-t)D^2v)>0,\qquad 0\le t\le 1.
	\]
	Thus in both cases the matrices $tD^2u+(1-t)D^2v$ lie in the elliptic region of $F$ for all $0\le t\le 1$. Since this region is open and the relevant sets are compact, there exist uniform ellipticity constants $\lambda,\Lambda>0$, independent of $R$, such that
	\[
	\lambda |\xi|^2 \le a^{ij}(x)\xi_i\xi_j \le \Lambda |\xi|^2
	\]
	for all $x\in\Omega_R\setminus\overline{U_\delta}$ and all $\xi\in\mathbb{R}^n$.
	
	Fix $\epsilon>0$. Since $w=0$ on $\partial D$, there exists $\delta>0$ such that
	\[
	|w|\le\epsilon \quad \text{on } \{x\in\mathbb{R}^n\setminus\overline{D}:d(x)=\delta\}.
	\]
	Since $w(x)=o(1)$ as $|x|\to\infty$, there exists $R>1$ such that
	\[
	|w|\le\epsilon \quad \text{on } \partial E_R.
	\]
	Applying the maximum principle to $w$ and $-w$ in the bounded domain
	\[
	\Omega_R\setminus\overline{U_\delta},\qquad U_\delta:=\{x\in\mathbb{R}^n\setminus\overline{D}:0<d(x)<\delta\},
	\]
	we obtain
	\[
	|w|\le\epsilon \quad \text{in } \Omega_R\setminus\overline{U_\delta}.
	\]
	Letting $R\to\infty$ and then $\epsilon\to0$ (hence $\delta=\delta(\epsilon)\to0$), we conclude $w\equiv0$ in $\mathbb{R}^n\setminus\overline{D}$. Hence $u\equiv v$.
	
	\section*{Acknowledgments}
	
	This work was partially supported by the National Natural Science Foundation of China 
	(Nos. 12171389 and 11801015).
	
	\section*{Conflict of Interest}
	The authors declare that there is no conflict of interest.

	\section*{Data Availability}
	No data were generated or analyzed in this study.

\end{document}